\documentclass[12pt,leqno]{amsart}
\usepackage{latexsym,amsmath,amsfonts,amssymb,cite}
\usepackage{color,latexsym,amsmath,amsthm,amsfonts,amssymb,xcolor,verbatim,relsize}

 \scrollmode

\newtheorem{thm}{Theorem}[section]
\newtheorem{deff}[thm]{Definition}

\newtheorem{rem}[thm]{Remark}
\newtheorem{prop}[thm]{Proposition}

\newtheorem{ex}[thm]{Example}
\newtheorem{claim}[thm]{Claim}

\newcommand{\wt}{\widetilde}
\newcommand{\wh}{\widehat}
\newcommand{\vO}{\varOmega}

\def\N{{\mathbb N}}
\def\R{{\mathbb R}}
\def\P{{\mathbb P}}
\def\mcL{{\mathcal L}}

\def\mcP{{\mathcal P}}
\def\mcA{{\mathcal A}}

\def\mcM{\mathcal M}
\def\mcN{\mathcal N}
\def\vL{{\varLambda}}
\def\mfA{{\mathfrak A}}
\def\mfB{{\mathfrak B}}
\def\mfC{{\mathfrak C}}
\def\mfD{{\mathfrak D}}

\def\mfL{\mathfrak L}

\def\mfZ{\mathfrak Z}

\def\emp{\emptyset}

\def\mcN{\mathcal N}

\def\tr{\triangle}
\def\cd{\circledast}
\def\di{\diamond}
\def\dd{\divideontimes}

\begin{document}

\title{Existence of measurable versions\\ of stochastic processes}
\author{Kazimierz Musia\l}
\address{Institut of Mathematics, Wroc{\l}aw University,  Pl. Grunwaldzki  2/4, 50-384 Wroc{\l}aw (Poland), Orcid Id: 0000-0002-6443-2043}
\email{kazimierz.musial@math.uni.wroc.pl}

\date{\today}



\begin{abstract}
Let $(X, \mfA,P)$, $(Y, \mfB,Q)$ be two arbitrary probability spaces and $\P:=\{(\mfA,P_y):y\in{Y}\}$ be a regular conditional probability (rcp) on $\mfA$ with respect to $Q$. Denote by $R$  the skew product of $P$ and $Q$ determined by  $\P$ on the product $\sigma$-algebra $\mfA\otimes\mfB$ and by $\wh{R}$ its completion. I prove  that  if   $(X, \mfA,P)$  is separable   in the Fr\'{e}chet-Nikod\'{y}m pseudo-metric, then the stochastic process  $\{\xi_y:y\in{Y}\}$ has an equivalent measurable modification  if and only if it is measurable with respect to   a certain particular $\sigma$-algebra larger than $\mfA\otimes\mfB$.  The theorem is  a strong generalization of \cite[Theorem 5.5]{mms2} and \cite[Theorem 6.1]{smm},where it was proved only that a suitable class of liftings transfer a measurable process into a measurable process.  It is known that not every process possesses an equivalent measurable modification (cf. \cite[Section 19.5]{St}).
  My approach is essentially different from earlier  trials. It reverts to \cite[Theorem 3]{ta1}, where Talagrand proved existence of an equivalent  separable modification of a measurable process (in case of $R=P\times{Q}$), provided $Y$ is endowed with a separable pseudometric.
  \end{abstract}

\maketitle

\section{Preliminaries}
Let $(X,\mfA,P)$ and $(Y,\mfB,Q)$ be fixed probability spaces. The algebra on $X\times{Y}$ generated by the rectangles $A\times{B}$, where $A\in\mfA$ and $B\in\mfB$, is denoted by  $\mfA\times\mfB$. The $\sigma$-algebra generated by $\mfA\times\mfB$ is denoted by $\mfA\otimes\mfB$. If $(Z,\mfZ,T)$ is a measure space, then $(Z,\wh\mfZ,\wh{T})$ denotes its completion and $\mfZ_0:=\{A\in\mfZ:T(A)=0\}$. $\mcL^{\infty}(T)$ is the collection of bounded $T$-measurable real functions. $\mfB_{\R}$ denotes the Borel subsets of the real line $\R$.
Assume that $\{P_y:y\in{Y}\}$ is a regular conditional expectation on $\mfA$ with respect to $Q$.  The measure $R$ on $(X\times{Y},\mfA\otimes\mfB)$ defined by the formula $R(E)=\int_YP_y(E^y)\,dQ(y)$ ($E^y:=\{x\in{X}:(x,y)\in{E}\}$)  is called a skew product of $P$ and $Q$.  We will use the notation $R\in{P\cd{Q}}$ ($P\cd{Q}$ denotes the family of all measures on $\mfA\otimes\mfB$ with $P$ and $Q$ as marginal measures).

There are several papers concerning existence of an equivalent measurable  modification of a stochastic process defined on a measure space (cf. \cite{Cohn, Cohn2, Dudley, HJ, ta3}). All are investigated for the direct product of marginal measures.  In particular, Cohn \cite{Cohn} presents  an example of a process and a lifting (assuming the continuum hypothesis)  that produce a  non-measurable  modification  of the measurable process (defined in case of ordinary product measure). The construction is too complicated to describe it here. A more subtle construction is presented in \cite{ta3}, where preservation of stability under liftings is investigated.  In \cite{Cohn2}  Cohn proves that the existence of an equivalent measurable modification is equivalent to the existence of a separable (in the classical sense) measurable  modification.  But he does not describe any general situation guaranteeing the existence of a measurable modification.  But he proves in \cite[Theorem2]{Cohn2} that each stochastically continuous process has a measurable separable  modification.   Hoffmann-J{\o}rgensen in \cite{HJ}  presented a necessary and sufficient condition for the existence of equivalent measurable process in terms of two dimensional distributions of the variables $(\xi_s,\xi_t)$.  It follows that the conditions from \cite{HJ} and that of mine are equivalent in case of separable $(X,\mfA,P)$, but they look completely different. It seems that a direct proof of their equivalence may be complicated.  In \cite{ta1} Talagrand proved that if $\{\xi_y:y\in{Y}\}$ is a stochastic process  and $Y$ is endowed with a separable pseudometric (no measure on $Y$ is required), then the measurable process possesses a separable modification (produced with the help of a special lifting).

I present a necessary and sufficient condition for the existence of a measurable  modification  of a process in the situation, when  $(X,\mfA,P)$ is separable    in the Fr\'{e}chet-Nikod\'{y}m pseudo-metric and $R$ is defined by a regular conditional probability (Theorem \ref{T2}).  It is known that an rcp not always exists, but it always exists if $P$ is a compact probability (see  \cite{Pachl1}).  Thus the generalization is of two kinds. Firstly, the product measure $P\times{Q}$ is replaced by an arbitrary skew product of $P$ and $Q$ determined by a regular conditional probability.  Secondly, I prove that even if some liftings do not produce measurable  modifications of the process (in spite of satisfying the necessary measurability assumptions required by Theorem \ref{T1}), then there is  another lifting (in fact the large family of liftings) that modifies the initial process into a measurable one.

 Classically, rather the direct product measure on $X\times{Y}$  is considered, but also in that case my characterization is new.  I require only a measurability of the investigated process with respect to some larger $\sigma$-algebra. It is the first necessary and sufficient purely measure theoretical condition guaranteeing the existence of measurable  modifications of an arbitrary stochastic process defined on an arbitrary separable probability space. That is the main novelty of the presented approach. Besides I present also one necessary and one sufficient condition for the existence of measurable  modification  in case of an arbitrary    $(X,\mfA,P)$ (Theorem \ref{T1}).

\begin{deff}\label{d1} \rm The family $\P:=\{(\mfA,P_y):y\in{Y}\}$ is called a  regular conditional probability (rcp) on $\mfA$ with respect to $Q$ if  for every $A\in\mfA$ the function $y\longrightarrow{P_y(A)}$ is $\wh{Q}$-measurable and
$$
\int_BP_y(A)\,dQ(y)={R}(A\times{B})\quad\mbox{for every }A\in\mfA\;\mbox{and }B\in\mfB\,.
$$
\end{deff}
\begin{deff}\rm
Let $(X,\mfA,P)$ and $(Y,\mfB,Q)$ be arbitrary probability spaces and $\P:=\{(\mfA,P_y):y\in{Y}\} $ be a regular conditional probability on $\mfA$ with respect to $Q$. . Moreover, let  $\varXi:=\{\xi_y:y\in{Y}\}$ and  $\varTheta:=\{\theta_y:y\in{Y}\}$ be stochastic processes defined on $(X,\mfA,\P)$ (that means that each $\xi_y$ and $\theta_y$  is a random variable defined on the $P_y$-completion of $(X,\mfA,P_y)$). We say that the processes are $\P$-equivalent if $\wh{P_y}\{x\in{X}:\xi_y(x)\neq\theta_y(x)\}=0$ for every $y\in{Y}$. The process $\{\theta_y:y\in{Y}\}$ is called measurable if the function $(x,y)\longrightarrow\theta_y(x)$ is $\wh{R}$-measurable. If $\varXi$ and $\varTheta$ are $\P$-equivalent and $\varTheta$ is measurable, then $\varTheta$ is called a measurable modification of $\varXi$.
\end{deff}
	Let $M:=\{(x,y)\in{X\times{Y}}\colon \xi_y(x)\neq\theta_y(x)\}$.  One knows that in general $M\notin\mfA\wh\otimes\mfB$, what means that $R^*(M)>0$.

If $\mfC\subset\mfA$ is a $\sigma$-algebra, then we define the following collection of sets:
\[
 \mcM_{\mfC}:=\{E\subset{X\times{Y}}\colon \exists\;N\in\mfB_0\;\forall\;y\notin{N}\;\wh{P_y|_{\mfC}}(E^y)=0\}.
\]
$\mcM_{\mfC}$ is called a family of  $R$-left nil sets  (compare with \cite[3.2.2, Satz 1, Satz 2]{Haupt}) with respect to $\mfC$.
Since for each $y\in{Y}$ the $\wh{P_y|_{\mfC}}$-zero sets form a $\sigma$-ideal  in $\mcP(X)$, the elements of $\mcM_{\mfC}$   forms a   $\sigma$-ideal in $\mcP(X\times{Y})$.

If $\mfD\supset\mfC$ is a $\sigma$-algebra, then $\mfC\dd_{\mfD}\mfB:=\{W\triangle{N}\colon W\in\mfC\otimes\mfB\;\;\&\; N\in\mcM_{\mfD}\}$ .  $\mfC\dd_{\mfD}\mfB$  is in general a $\sigma$-algebra larger than $\mfC\otimes\mfB$.

If $V\in\mfA\dd_{\mfA}\mfB$, then the formula
\[
R_{\dd}(V)=\int_Y\wh{P_y}(V^y)\,d\wh{Q}(y)
\]
 defines an extension of $R$.
One can  check that if $V=W\tr{N}$ ($W\in\mfA\otimes\mfB$ and $N\in{\mcA_{\mfA}}$), then  $R_{\dd}(W\triangle{N})=R(W)$.
The function $y\to {V^y}$ is $\wh{Q}$-measurable, because the function  $y\to{W^y}$ is $\wh{Q}$-measurable.

 $R_{\dd}$ is a complete measure and $\mcM_{\mfA}$ is the $\sigma$-ideal of $R_{\dd}$-zero sets.

$\mfA\wh\otimes\mfB$ - the completion of $\mfA\otimes\mfB$ with respect to $R$ - is contained in $\mfA\dd_{\mfA}\mfB$ and $R_{\dd}$ is an extension of $\wh{R}$.

Let
\[
\mcN_{\mfC}:=\{E\subset{X\times{Y}}\colon\forall\;y\in{Y}\;\wh{P_y|_{\mfC}}(E^y)=0\}.
\]
  The new family seems to illustrate  better  equivalent relation of the processes.
One can see immediately that processes  $\varXi$ and $\varTheta$ are equivalent if and only if there exists a set $N\in\mcN_{\mfA}$ such that $\varXi=\varTheta$ on $N^c$. Thus, the $\sigma$-ideal $\mcM_{\mfA}$ is implicitly contained in the definition of the equivalence. Unfortunately, the $\sigma$-algebras defined by  $\mfC\di_{\mfC}\mfB:=\{W\triangle{N}\colon W\in\mfC\otimes\mfB\;\;\&\; N\in\mcN_{\mfC}\}$ are usually not complete with respect to $(R_{\dd})|_{\mfC\di_{\mfC}\mfB}$.  That excludes them from Theorem \ref{I}, where the target $\sigma$-algebra must be complete. The proper choice remain $\sigma$-algebras  $\mfC\dd_{\mfC}\mfB$.

The following theorem is the key result in our examination of measurability of stochastic processes. Lifting is defined in the classical way as in \cite{it}. The collection of all liftings on $(X,\mfA,P)$ is denoted by $\vL(P)$.
\begin{thm}\label{I}
Assume that ${\mathfrak A}$ contains a countably generated
$\sigma$-algebra which is dense in ${\mathfrak A}$ (in the
Fr\'{e}chet-Nikod\'{y}m pseudo-metric) with respect to $P$.  Moreover, let $\{(\mfA,P_y):y\in{Y}\}$ be an rcp on $\mfA$ with respect to $Q$.  Then for every $y\in Y$
there exists $\sigma_y\in\vL(\wh{P}_y)$  and there exists a lifting  $\pi:\mcL^{\infty}(R_{\dd})\to \mcL^{\infty}(\wh{R})\subset\mcL^{\infty}(R_{\dd})$  such that
\begin{equation}\label{f2}
[\pi(f)]^y = \sigma_y\Bigl([\pi(f)]^y\Bigr) \qquad\mbox{for all}
\quad y\in Y\quad\mbox{and}\quad f \in {\mcL}^{\infty}(R_{\dd}).
\end{equation}
\end{thm}
\begin{proof}
According to \cite[Theorem 3.6]{smm} for every $y\in Y$
there exists $\sigma_y\in\vL(\wh{P}_y)$ and there exists a lifting $\pi:\mcL^{\infty}(\wh{R})\to\mcL^{\infty}(\wh{R})$  satisfying \eqref{f2}  with $R_{\dd}$ replaced by $\wh{R}$. The liftings are constructed step-by-step on an arbitrary (but fixed) countably generated $\sigma$-algebra that is dense in $\mfA$ and then extended to the domains of $\wh{P_y}$ and onto $\mfA\wh\otimes\mfB$.  Since in the whole lifting theory axiom of choice plays an important role, the construction is not unique.

 $\mfA\wh\otimes\mfB$ is $R_{\dd}$-dense in $\mfA\dd_{\mfA}\mfB$ and so if $E\in\mfA\dd_{\mfA}\mfB$ and $F\in\mfA\wh\otimes\mfB$  are such that $R_{\dd}(E\tr{F})=0$, then we may set $\wt\pi(E):=\pi(F)$. Then $\wt\pi:\mfA\dd_{\mfA}\mfB\to\mfA\wh\otimes\mfB$ is a lifting satisfying \eqref{f2}. We denote it also by  $\pi$.
\end{proof}
\begin{rem}\rm
The thesis of Theorem \ref{I} remains valid if  the separability of $\mfA$ is replaced by the absolute continuity of every $P_y$ with respect to $P$ or by the absolute continuity of $R$ with respect to $P\times{Q}$ (see \cite{mms1}).  But each of the last two assumptions (that are in some sense almost equivalent) would be too restrictive for our main result. Theorem \ref{T2} does not require any absolute continuity.  The assumption of absolute continuity would exclude from our investigation a large class of processes with not absolutely continuous rcp.
To the best of my knowledge no necessary and sufficient condition guaranteeing validity of \eqref{f2} is known.
\end{rem}

\section{Measurability of processes}\label{A}
There are known several examples of stochastic processes without equivalent measurable  modifications and several conditions guaranteeing existence of measurable modifications, all in case when $R=P\times{Q}$ (cf. \cite{Cohn, Dudley, HJ,ta3}). Below I suggest a different purely measure theoretical approach, based on existence of  special liftings in product measure spaces endowed with a regular conditional probability.
\begin{thm}\label{T1}
Let $\varXi:=\{\xi_y:y\in{Y}\}$ be a stochastic process defined on $(X,\mfA,\P)$, where $\P=\{(\mfA,P_y):y\in{Y}\}$ is an rcp on $\mfA$ with respect to $Q$. Moreover, let $R\in{P}\cd{Q}$ be determined by $\P$. Then,  the following statements hold true:
\begin{itemize}
\item[(M1)]
If the process $\{\xi_y:y\in{Y}\}$  possesses a $\P$-equivalent measurable modification, then there exists a countably generated $\sigma$-algebra $\mfC\subset\mfA$ such that    $\varXi$ is measurable with respect to $\mfC\dd_{\mfA}\mfB$;
\item[(M2)]
If there exists a $\sigma$-algebra $\mfC\subset\mfA$ such that  $\mfC$ is separable  in the Fr\'{e}chet-Nikod\'{y}m pseudo-metric  and  $\varXi$ is measurable with respect to $\mfC\dd_{\mfC}\mfB$, then  the process $\{\xi_y:y\in{Y}\}$  possesses a $\P$-equivalent measurable modification.
\end{itemize}
If the process  $\varXi$ is bounded and the measurable modification $\{\theta_y:y\in{Y}\}$ is built with the help of the liftings $\pi$ and $\{\sigma_y:y\in{Y}\}$ taken from Theorem \ref{I}, then for every $y\in{Y}$ the equality $\theta_y=\sigma_y(\xi_y)$ is fulfilled.
\end{thm}
\begin{proof}
$(M1)$ \quad Let $\varTheta(x,y):=\theta_y(x)$ be $\wh{R}$-measurable and equivalent to $\{\xi_y:y\in{Y}\}$.  Define the set $N$ by the equality
\[
N:=\{(x,y)\in{X\times{Y}}\colon \theta_y(x)\neq\xi_y(x)\}.
\]
 By the assumption $\wh{P_y}(N^y)=0$ for every $y\in{Y}$. Thus $N\in\mcM$.  If $B\subset\R$ is a Borel set, then
\[
\varXi^{-1}(B)=\varXi^{-1}(B)\cap{N^c}\cup\varXi^{-1}(B)\cap{N}=\varTheta^{-1}(B)\cap{N^c}\cup\varXi^{-1}(B)\cap{N}
\]
 and consequently $\varXi^{-1}(B)\triangle\varTheta^{-1}(B)\subset{N}\in\mcM$. That  means that $\{\xi_y:y\in{Y}\}$ is measurable with respect to $\sigma(\varTheta^{-1}(\mfB_{\R})\cup\mcM)$. But $\varTheta^{-1}(\mfB_{\R})$ is a countably generated $\sigma$-algebra what immediately yields the existence of a countably generated $\sigma$-algebra $\mfC\subset\mfA$ satisfying the inclusion  $\varTheta^{-1}(\mfB_{\R})\subset \mfC\otimes\mfB$.  Clearly, $\varXi$ is measurable with respect to $\mfC\dd_{\mfA}\mfB$.

$(M2)$ \quad    Assume at the beginning that the process defined by $\varXi(x,y):=\xi_y(x)$ is bounded, i.e. there exists $\alpha\in(0,+\infty)$ such that $|\xi_y(x)|\leq\alpha$, for every $(x,y)\in{X}\times{Y}$. For simplicity of the notation assume also that $\mfC=\mfA$.

According to Theorem \ref{I} there exist a lifting $\pi:\mcL^{\infty}(R_{\dd})\to \mcL^{\infty}(\wh{R})$ and liftings $\{\sigma_y\in\vL(\wh{P_y}):y\in{Y}\}$ such that $[\pi(f)]^y = \sigma_y\bigl([\pi(f)]^y\bigr)$ for every $f \in {\mcL}^{\infty}(R_{\dd})$  and every $y\in Y$.

Let $\varXi$ be $R_{\dd}$-measurable. $\varXi:X\times{Y}\to\R$ is a function of two variables. In virtue of Theorem \ref{I}  the function $\pi(\varXi)$ is $\wh{R}$-measurable and
\[
\sigma_y\bigl([\pi(\varXi)]^y\bigr)=[\pi(\varXi)]^y \quad\mbox{for every }y\in{Y}.
\]
 It follows from the elementary properties of lifting  that there exists a set $N_1\in\mfB_0$ such that  for every $y\notin{N_1}$
\[
\sigma_y\bigl([\pi(\varXi)]^y\bigr)=[\pi(\varXi)]^y=\sigma_y(\varXi^y)=\sigma_y(\xi_y)\quad\wh{P_y}-a.e.
\]
It follows that $\varTheta$ defined by $\varTheta(\cdot,y):=\sigma_y(\xi_y) $ for every $y\in{Y}$, is $\wh{R}$-measurable and $\P$-equivalent to $\varXi$. More precisely, $\varTheta|_{X\times{N_1^c}}=\pi|_{X\times{N_1^c}}$.

If the process $\varXi$ is unbounded and $\varXi$ is $\mfA\dd_{\mfA}\mfB$-measurable, then one can decompose the space $X\times{Y}$ into at most countably many pairwise disjoint sets $\vO_n\in\mfA\dd_{\mfA}\mfB$ such that $\varXi$ is bounded on each $\vO_n,\,n\in\N$. One may set for instance $\vO_n:=\{(x,y):n\leq|\varXi(x,y)|<{n+1}\}$. Let for each $n\in\{0,1,2,\ldots\}$
\[
\varXi_n(x,y):=\varXi(x,y)\,\chi_{\vO_n}(x,y).
\]
 On the strength of  the proof for bounded processes, for each $n$ there exists a $\wh{R}$-measurable
\[
\varTheta_n:=\pi\bigl(\varXi_n\bigr)=\pi\bigl(\varXi_n\bigr)\chi_{\pi(\vO_n)}
\]
 that is equivalent to $\varXi_n$.  Since the supports of the functions $\varTheta_n$ are pairwise disjoint (if $n\neq{m}$, then $\pi(\vO_n)\cap\pi(\vO_m)=\emp$) and $\wh{R}\bigl(\bigcap_n\pi(\vO_n^c)\bigr)=0$, we may define the process $\varTheta$ setting
\[
\varTheta(x,y) = \left\{ \begin{array}{ll}
\varTheta_n(x,y)&{\rm\; if} \ \  (x,y)\in\pi(\vO_n)\\
&\\
\xi_y(x)&{\rm\; if} \ \  (x,y)\in\bigcap_n\pi(\vO_n^c)\ .   \\
\end{array}\right.
\]
Let  $N:=\{y:\wh{P_y}\bigl(\bigl[\bigcap_n\pi(\vO_n^c)\bigr]^y\bigr)=0\}$. Then $\wh{Q}(N)=1$ and   $\varTheta^y\equiv_{\wh{P_y}}\xi_y$ for $y\in{N}$. If $y\notin{N}$ and $A:=\bigl[\bigcap_n\pi(\vO_n^c)\bigr]^y$, then $\varTheta(x,y)=\varTheta_n(x,y)$ if $x\notin{A}$ and $ (x,y)\in\pi(\vO_n)$ and,  $\varTheta(x,y)=\xi_y(x)$ in case of $x\in{A}$ and $  (x,y)\in\pi(\vO_n)$.  All measurability requirements and consistency properties follow from the analogical properties of each $\varTheta_n$.
\end{proof}
\begin{thm}\label{T2}
Let $\varXi:=\{\xi_y:y\in{Y}\}$ be a stochastic process defined on $(X,\mfA,\P)$, where $\P=\{(\mfA,P_y):y\in{Y}\}$ is an rcp on $\mfA$ with respect to $Q$. Moreover, let $R\in{P}\cd{Q}$ be determined by $\P$. If $(X,\mfA,P)$ is separable in the Fr\'{e}chet-Nikod\'{y}m pseudo-metric, then  the process  $\{\xi_y:y\in{Y}\}$ has an equivalent measurable modification if and only if it is measurable with respect to  $\mfA\dd_{\mfA}\mfB$.
\end{thm}
\begin{rem} \rm If the process  $\{\xi_y:y\in{Y}\}$ has an equivalent measurable modification, then the modification is  separable in the Fr\'{e}chet-Nikod\'{y}m pseudo-metric (in virtue of $(M1)$). But that does not mean the separability in the classical sense (cf. \cite{ta1}). It follows that the separability condition -- as it is formulated in (M1) -- is  indispensable.  However, it is probably not sufficient in case of non-separable measure spaces. Thus, some kind of separability  reflects a deeper property of the class of processes admitting measurable modifications. But accidentally it is also a technical assumption required by the lifting argument (without the separability assumption Theorem \ref{I} may be false).
\end{rem}
\begin{rem} \rm
At this place it is important to notice that the liftings $\sigma_y$ are not quite arbitrary, they are chosen in a very special way. These   are so called admissibly generated liftings (for details see \cite{smm}). Without such an assumption the lifting modification of a measurable process may produce a non measurable process (see \cite{Cohn,Dudley}).
\end{rem}

 In the most classical version $R=P\times{Q}$ and  processes  $\{\xi_y:y\in{Y}\}$ and  $\{\theta_y:y\in{Y}\}$ are called equivalent if $P(\{\xi_y\neq\theta_y\})=0$ for every $y\in{Y}$.  In that case the rcp is almost uniquely defined ($P_y=P$ for $Q$-a.e. $y\in{Y}$) and one may assume that   $P_y=P$ for every $y\in{Y}$. Hence  $\mcM_{\mfC}=\{E\subset{X}\times{Y}:\exists\;N\in\mfB_0\;\forall\;y\notin{N}\;\wh{P|_{\mfC}}(E^y)=0\}$.

In the language of functions of two variables, Theorem \ref{T2} can be formulated in the following way:

\begin{claim}  Let   $(X,\mfA,P)$ be separable in the Fr\'{e}chet-Nikod\'{y}m pseudo-metric. A function $f:X\times{Y}\to\R$ with measurable $Y$-sections has an equivalent $\mfA\wh\otimes\mfB$-measurable  modification if and only if  $f$ is measurable with respect to $\mfC\dd_{\mfA}\mfB$.
\end{claim}
\begin{ex}\rm
Let $(X,\mfA,P)=(Y,\mfB,Q)=([0,1],\mfL\cap[0,1],\lambda)$, where $\lambda$ is the Lebesgue measure on the $\sigma$-algebra $\mfL$ of Lebesque measurable sets. Moreover, let $V\subset{Y}$ be a Vitali set. For each $y\in{V}$ choose a non-trivial $\lambda$-zero set $A_y\in\mfB_{\R}$. Define the set $E\subset[0,1]^2$ in the following way: if $y\in{V}$, then $E^y:=A_y$ and if $y\notin{V}$, then $E^y=\emp$.  The function  $\chi_{{}_E}$ is $\mfL\dd\mfL$ measurable but not necessarily  plane-measurable.  It depends on the position of the sets $A_y$. The zero function is the measurable function that is equivalent to $\chi_{{}_E}$.  The concrete non-trivial example of $E$ can be found in\cite{S}, where -- under the continuum  hypothesis --
Sierpi\'nski  constructed a set $E\subset [0,1]\times[0,1]$  with
the following properties: \vskip5pt (1) all vertical sections $E^c_x$ are at most countable,

(2) all horizontal sections $E^y$ are at most countable.

$E\in\mfL\dd_{\mfL}\mfL$ (more precisely $E\in\mcM_{\mfL}$) but $E\notin \mfL\wh\otimes\mfL$.

In a more general case,  assume that there exists  a function $h:X\to{Y}$ such that its graph $G$ is not in $\mfA\wh\otimes\mfB$.  Then $G\in\mfA\dd_{\mfA}\mfB$. If $\{\xi_y:y\in{Y}\}$ is a measurable process, them $\varTheta$ defined by $\varTheta(x,y)=\xi_y(x)$ for $(x,y)\notin{G}$ and  $\varTheta(x,y)=0$ for $(x,y)\in{G}$ is the process that is  $\mfA\dd_{\mfA}\mfB$-measurable and admits a measurable modification.
\end{ex}
\begin{ex}\rm
Cohn \cite{Cohn}, assuming the continuum hypothesis, constructed a measurable process $\varXi$  such that for some lifting $\rho$ on the completion of $(X,\mfA,P)$ the process $\varTheta$ defined by $\varTheta(x,y):=\rho[\varXi(\cdot,y)](x)$ is not measurable. The construction is too complicated even for a short description. It is not based on any product lifting. The process $\varTheta$ is a non-trivial example of a non-measurable process possessing a measurable modification. It follows from the condition (M1) of Theorem \ref{T1} that $\varTheta$ is measurable with respect to some $\mfC\dd_{\mfA}\mfB$ with separable $\mfC$.
\end{ex}
\begin{ex}  \rm
Applying Theorem \ref{T2} we present a simple example of a stochastic process that has no measurable modification. The example is copied from \cite[Section 19.5]{St}. Let $(X,\mfA,P)$ be arbitrary separable probability space, $(Y,\mfB,P)=([0,1],\mfL,\lambda)$ and $R=P\times\lambda$.  Let $\varXi$ be the process defined by $\varXi(x,y)=\xi_y(x)$, where the random variables $\{\xi_y:y\in{Y}\}$ are mutually independent, bounded and such that
\begin{equation}\label{E1}
\mathbb{E}(\xi_y)=0\quad\mbox{and  }\quad\mathbb{E}(\xi_y^2)=1\quad\mbox{for every }y\in{Y}.
\end{equation}
 Suppose that $\varXi$ is $\mfA\dd\mfL$-measurable. Then, according to Theorem \ref{T2} there exists a lifting $\pi:\mcL^{\infty}(R_{\dd})\to\mcL^{\infty}(\wh{R})$ such that $\varTheta$ defined by $\varTheta:=\pi(\varXi)$ is a measurable modification of $\varXi$. But the new variables $\theta:=\sigma_y(\xi_y)$ are also mutually independent and satisfy \eqref{E1} with $\xi_y$ replaced by $\theta_y$. Applying the Fubini theorem we obtain
for each closed interval $I\subset[0,1]$ with rational endpoints
\[
{\mathbb E}\left\{\left(\int_I\theta_y(x)\,dy\right)^2\right\}={\mathbb E}\left\{\int_I\int_I\theta_y(x)\theta_z(x)\,dy\,dz\right\}=\int_I\int_I{\mathbb E}(\theta_y\theta_z)\,dy\,dz=0.
\]
That means that for each $x\notin{N_I}\in\mfA_0$ we have $\int_I\theta_y(x)\,dy=0$.  If $N=\bigcup_IN_I$, then $P(N)=0$ and $\int_J\theta_y(x)\,dy=0$ for every interval $J\subset[0,1]$ and every $x\notin{N}$. That means that for each $x\notin{N}$ we have $\theta_y(x)=0\;\;\lambda$-a.e.  Applying the Fubini theorem, we have
\[
\int_X\int_0^1\theta_y(x)^2\,dP(x)\,dy=0.
\]
On the other hand
\[
\int_X\int_0^1\theta_y(x)^2\,dP(x)\,dy=\int_0^1{\mathbb E}(\theta_y^2)\,dy=1.
\]
Thus, $\varTheta$ is not measurable. It follows from Theorem \ref{T2} that $\varXi$ is not measurable with respect to $\mfA\dd_{\mfA}\mfB$.
\end{ex}
\section{Iterated integrals.}
In case of the direct product the coordinate measure spaces and non-measurable process applying the Fubini theorem is impossible. Instead, following \cite{BM},  I remind a generalization of the Fubini theorem covering the case of separately integrable functions.

In case of the direct product the coordinate measure spaces are equally important. From that point of view it is suitable to introduce the $\sigma$-ideals
\[
 \mcM_{\mfA}^{\perp}:=\{E\subset{X\times{Y}}\colon \exists\;N\in\mfA_0\;\forall\;x\notin{N}\;\wh{Q}(E_x)=0\},
\]
\begin{equation*}
\mcM^{\bigstar}:=\mcM_{\mfA}\cap{\mcM_{\mfA}^{\perp}}\\
=\left\{ E\subset {X\times{Y}}\,\left|\
\begin{aligned}
&
\exists\;N\in\mfA_0\;\exists\;M\in\mfB_0\\[-1mm]
& \forall\;x\notin{N}\;\wh{Q}(E_x)=0\;\forall\;y\notin{M}\;\wh{P}(E^y)=0
\end{aligned}
\right.\right\}\,.
\end{equation*}
and the $\sigma$-algebras $\mfA\bigstar\mfB:=\sigma((\mfA\otimes\mfB)\cup{M^{\bigstar}})$ and $\mfA\dd^{\perp}\mfB:=\sigma\bigl((\mfA\otimes\mfB)\cup\mcM^{\perp}\bigr)$. The $\sigma$-ideal $\mcM^{\bigstar}$ was introduced in \cite{BM} and was called the $\sigma$-ideal of nil sets.  Let  $R^{\dd}:\mfA\dd^{\perp}\mfB\to[0,1]$ be defined by $R^{\dd}(F)=\int_X\wh{Q}(F_x)\,d\wh{P}(x)$. If $F=E\tr{M}$ with $E\in\mfA\wh\otimes\mfB$ and $M\in\mcM^{\bigstar}$, then
\begin{align*}
&R_{\dd}(F)=\int_Y\wh{P}(F^y)\,d\wh{Q}(y)=\int_Y\wh{P}(E^y)\,d\wh{Q}(y)=(P\wh\otimes{Q})(E)\\
&=\int_X\wh{Q}(E_x)\,d\wh{P}(x)=\int_X\wh{Q}(F_x)\,d\wh{P}(x)=R^{\dd}(E)
\end{align*}
It follows that the measure $P\wh\times{Q}$ can be extended to $\mfA\bigstar\mfB$ by setting $R_{\bigstar}(F)=(P\wh\times{Q})(E)$. Moreover,
$R_{\bigstar}(F)=R_{\dd}(F)=R^{\dd}(F)$ for every $F\in\mfA\bigstar\mfB$. The above equalities yield immediately the following interesting generalization of the Fubini theorem, that can be found in \cite{BM} but now it seems to be forgotten. Let us say that $f$ is nearly separately measurable (integrable) if $f^y$ is $\mfA$-measurable ($P$-integrable) for $Q$-a.e. $y\in{Y}$ and $f_x$ is $\mfB$-measurable ($Q$-integrable) for $P$-a.e. $x\in{X}$.
\begin{prop}\label{p1}
Let $(X,\mfA,P)$,  $(Y,\mfB,Q)$ be complete probability spaces and $f:X\times{Y}\to\R$ be a function.
If $f$ is $R_{\bigstar}$-integrable or non-negative, then it is nearly separately integrable and
\begin{align*}
&\int_{X\times{Y}}f\,dR_{\bigstar}=\notag\\
&=\int_Y\Bigl(\int_Xf(x,y)\,dP(x)\Bigr)\,dQ(y)=\int_X\Bigl(\int_Yf(x,y)\,dQ(y)\Bigr)\,dP(x)\,.
\end{align*}
\end{prop}
\begin{proof}
Assume that $f$ is bounded and a uniform limit of a sequence of simple functions $f_n=\sum_{i=1}^{k_n}a_{i,n}\chi_{{}_{E_{k,n}}}$, where $E_{k,n}\in\mfA\bigstar\mfB$. For every $n\in\N$ we have the equality
\begin{align*}
&\int_{X\times{Y}}f_n\,dR_{\bigstar}=\sum_{i=1}^{k_n}a_{i,n}\int_{X\times{Y}}R_{\bigstar}(E_{k,n})\\
&=\sum_{i=1}^{k_n}a_{i,n}R_{\dd}(E_{k,n})=\int_Y\Bigl(\int_Xf_n(x,y)\,dP(x)\Bigr)\,dQ(y)\\
&=\int_X\Bigl(\int_Yf_n(x,y)\,dQ(y)\Bigr)\,dP(x)
\end{align*}
Applying the Lebesgue Dominated Convergence Theorem, we obtain the equality
\begin{align*}
&\int_{X\times{Y}}f\,dR_{\bigstar}=\lim_n \int_{X\times{Y}}f_n\,dR_{\bigstar}\\
&=\lim_n\int_Y\Bigl(\int_Xf_n(x,y)\,dP(x)\Bigr)\,dQ(y)=
\int_Y\Bigl(\int_Xf(x,y)\,dP(x)\Bigr)\,dQ(y)
\end{align*}
Similarly for reversed order of $P$ and $Q$.
\end{proof}
\begin{rem} \rm
Assume that $R=P\times{Q}$ and $\varXi=\{\xi_y:y\in{Y}\}$ is a process with measurable paths $x\to{\xi_y(x)}$. If $\varXi$ is not measurable, then applying  the Fubini theorem is impossible. In such a case  Proposition \ref{p1} may be useful.
\end{rem}

\end{document}